\documentclass[a4paper, 11pt, reqno]{article}

\usepackage{cite}
\usepackage{tikz}\usetikzlibrary{matrix,decorations.pathreplacing,positioning}
\usepackage{amsmath,amsthm,amssymb,mathabx,bm,relsize,wasysym}
\usepackage{enumitem}
\usepackage{calrsfs}

\usepackage{setspace}
\usepackage{xcolor}
\definecolor{light-gray}{gray}{0.95}

\usepackage{hyperref}

\usepackage{tabularx}

\usepackage[margin=3cm]{geometry}

\definecolor{myg}{RGB}{220,220,220}
 
\theoremstyle{definition}
\newtheorem{theorem}{Theorem}[section]

\newtheorem{lemma}[theorem]{Lemma}

\newtheorem{example}[theorem]{Example}

\newtheorem{remark}[theorem]{Remark}

\newcommand*{\myproofname}{Proof of the claim}

\newcommand{\Mod}[1]{\ (\mathrm{mod}\ #1)}

\newlength{\mynodespace}
\title{\textbf{A New Series for Rogers-Ramanujan-Gordon Identities when $\mathbf{k=3}$}}

\usepackage{authblk}

\author{Yal\c{c}\i n Can K\i l\i\c{c} \\ Sabanc{\i} University, T\"urkiye\thanks{This work is partially supported by TUBITAK through grant 122F136. \\ Email: yalcinkilic@sabanciuniv.edu}}

\date{}
\begin{document}

\maketitle

\begin{abstract}
    In this paper, we introduce a new series of Rogers-Ramanujan-Gordon partitions when $k=3$. The combinatorial interpretation of the series is given by base partition, forward moves and backward moves. We conclude the paper with future research questions related to the generalization of this approach.
\end{abstract}

\section{Introduction} \label{sec:int}

The systematic study of \textit{partitions} started with Euler in \cite{euler2012introduction}.
An \emph{integer partition} of a natural number $n$ is a sequence of positive integers $(\lambda_1,\lambda_2,..,\lambda_k)$ where~$\lambda_1 \ge \lambda_2 \ge \cdots \ge \lambda_k \ge 1$ and $\lambda_1+\lambda_2+ \cdots+\lambda_k = n$.
We, usually, denote this partition as $\lambda_1+\lambda_2+\ldots+\lambda_k$. Each $\lambda_i$ is called a \emph{part}. We define $p(n)$ as the number of partitions of  $n$, and $p(n)$ is called the \emph{partition function}. For instance, all partitions of $4$ can be listed as: $4$, $3+1$, $2+2$, $2+1+1$, $1+1+1+1$.
Thus, $p(4)=5$.

Instead of looking at all possible partitions of a number, we can put some restrictions on the partition. If we impose conditions on the partitions, we denote the number of partitions of $n$ which satisfy the conditions as $p(n \, | \, \text{conditions})$. 
As an example, we can list all partitions of $5$ where each part is odd: $5$, $3+1+1$ and $1+1+1+1+1$. We can denote this as: $p(5 \, | \, \text{each part is odd}) = 3$.
Similarly, all partitions of $5$ where each part is distinct, i.e., we are not allowed to use repeated parts can be listed as: $5$, $4+1$ and $3+2$. As a result, $p(5 \, | \, \text{parts do not repeat}) = 3$. It is not a coincidence that~$p(5 \, | \, \text{odd parts}) = p(5 \, | \, \text{distinct parts})$. It follows from the following theorem of Euler:

\begin{theorem}[Euler's Theorem, \cite{andrews1998theory}]
Let $n$ be any natural number. Then, the number of partitions of $n$ where each part is odd is equal to the number of the partitions of $n$ where each part appears exactly once. In other words, $p(n \, | \, \text{each part is odd}) = p(n \, | \, \text{parts do not repeat})$.    
\end{theorem}

Rogers and Ramanujan found two partition identities which are very influential in the partition theory:

\label{thm:RR1}
\begin{theorem}[Rogers-Ramanujan 1, \cite{andrews1998theory}]

Let $n$ be any natural number. Then, $$p(n \, | \,  \text{parts} \equiv \pm1 \Mod{5}) = p(n \, | \, \text{Repeated parts and consecutive parts are not allowed}).$$
\end{theorem}

\label{thm:RR2}
\begin{theorem}[Rogers-Ramanujan 2, \cite{andrews1998theory}]

Let $n$ be any natural number. Then, $p(n \, | \, \text{parts}$ ~$\equiv \pm 2  \pmod 5) = p(n \, | \,$ Repeated parts and consecutive parts are not allowed, 1 cannot be used as a part.$)$
\end{theorem}

Let $n = 9$. Then, 
at the modulus condition side of Rogers-Ramanujan 1 we have $9$, $6 + 1 + 1 + 1$, $4 + 4 + 1$, $4 + 1 + 1 + 1 + 1 + 1$, $1 + 1 + 1 + 1 + 1 + 1 + 1 + 1 + 1$. At the difference condition side of Rogers-Ramanujan 1 we have $9$, $8 + 1$ , $7 + 2$, $6 + 3$, $5 + 3 + 1$.
Similarly, at the modulus condition side of Rogers-Ramanujan 2 we have $7+2$, $3 + 3 + 3$, $3 + 2 + 2 + 2$. At the difference condition side of Rogers-Ramanujan 2 we have $9$ , $7 + 2$ , $6+3$.

These type of identities are generalized as follows: Let $n$ be any natural number. Then identities of the type $$p(n \, | \, \text{modulus conditions on parts}) = p(n \, | \, \text{difference conditions on parts}),$$ are called \emph{Rogers-Ramanujan type identities}. Gordon generalized these identities in \cite{GORDON61} as follows:

\begin{theorem}[Rogers-Ramanujan-Gordon Identities]
\label{RRG}
Let $a$ and $k$ be natural numbers such that $1 \le a \le k$. Then, the number of partitions of $n$ into parts not equivalent to~$0$,~$\pm a \pmod {2k+1}$  is equal to the number of partitions of $n = \lambda_1 + \lambda_2 + ... + \lambda_k$ where~$\lambda_i \ge \lambda_{i+k-1} + 2$ and the number of 1's are at most $a-1$. 
\end{theorem}

Note that $(a,k)=(2,2)$ corresponds to Rogers-Ramanujan 1, and $(a,k)=(1,2)$ corresponds to Rogers-Ramanujan 2.

Let $(a,k)=(2,3)$ and $n=9$. Then, the modulus condition side of Theorem \ref{RRG} gives us: $8+1$, $6+3$, $6+1+1+1$, $4+4+1$ , $4+3+1+1$, $4+1+1+1+1+1$, $3+3+3$, $3+3+1+1+1$, $3+1+1+1+1+1+1$, $1+1+1+1+1+1+1+1+1$. At the difference condition side we have  $9$, $8+1$, $7+2$, $6+3$, $6+2+1$, $5+4$, $5+3+1$, $5+2+2$, $4+4+1$, $4+3+2$.

We will use the notation $rrg_{k,a}(m,n)$ to denote the number of partition of $n$ into $m$ parts which satisfy the modulus condition side of \ref{RRG}. If the numbers $m$, $n$ and $a$ are not taken into account, then we use $rrg_{k}$ to denote the partitions which satisfies the difference condition side of \ref{RRG}.

Andrews found an analytic version of Rogers-Ramanujan-Gordon identities:

\begin{theorem}[Andrews's Analytic Version of Rogers-Ramanujan-Gordon Theorem, \cite{andrews1974analytic}]
    Let $1 \le a \le k$ be integers. Then, $$\sum_{n_1,n_2\ldots,n_{k-1} \ge 0} \frac{q^{N_1^2+N_2^2+\ldots+N_{k-1}^2+N_a+N_{a+1}+\ldots+N_{k-1}}}{(q;q)_{n_1}(q;q)_{n_2}...(q;q)_{n_{k-1}}} = \prod_{\substack{n=1\\ n \nequiv 0, \, \pm a \, \Mod{2k+1}}} ^{\infty}\frac{1}{1-q^n}$$
    where $N_i := n_i+n_{i+1}+...+n_{k-1}$.
\end{theorem}

This identity is called \textit{Andrews-Gordon identity}. When $k=3$, Andrews's series becomes the following:
$$\sum_{m,n \ge 0} \frac{q^{(m+n)^2+n^2}}{(q;q)_m(q;q)_n}$$

Our aim is the present a new series for Rogers-Ramanujan-Gordon identity for the case~$k=3$. We denote the multivariate generating function of $rrg_{3,a}$ by $T_a(x)$. More precisely,

$$T_a(x):=\sum_{m,n \ge 0} rrg_{3,a}(m,n)x^mq^n$$

The main results of the paper are put together in the next theorem. We will give combinatorial explanations of them, as well. The combinatorial interpretation of the given series is done by series of moves, following the approaches in \cite{bressoud1980analytic}, \cite{kurcsungoz2010parity} and \cite{kurcsungoz2019andrews}. Moreover, it is possible to see our series from linked partition ideals perspective as in \cite{andrews1974general}, \cite{chern2019linked} and \cite{chern2020linked}. However, the latter approach is not studied in this paper.
\begin{theorem}
\label{thm:yck}

$$T_1(x) = \sum_{n,m \ge 0} \frac{q^{4\binom{m+1}{2}+2mn+\binom{n+1}{2}+n}x^{2m+n}}{(q^2;q^2)_m(q;q)_n}.$$

$$T_2(x) = \sum_{n,m \ge 0} \frac{q^{4\binom{m+1}{2}+2mn+\binom{n+1}{2}}x^{2m+n}}{(q^2;q^2)_m(q;q)_n}.$$

$$T_3(x) = \sum_{n,m \ge 0} \frac{q^{4\binom{m+1}{2}+2mn+\binom{n+1}{2}-2m}x^{2m+n}}{(q^2;q^2)_m(q;q)_n}.$$

\end{theorem}

As we will see in the rest of the paper, we need to consider cases $a=3$ and $a=1$ together, and $a=2$ seperately. We explain the reason for this seperation in the following sections.

The layout of the paper is as follows. In Section \ref{sec:examples} we discuss the intuitive idea of our approach and give examples. Section \ref{sec:main} is devoted to the main theorem of the paper and its proof. We finish the paper with possible avenues for further research in Section \ref{sec:Conc}.

\section{Intuitive Idea and Two Key Examples}
\label{sec:examples}
In this section we explain the intuitive idea and give some examples before discussing the details and giving the proof of the main theorem. Our intuitive idea is as follows. Given a partition satisfying $rrg_{3,a}$ conditions, divide the parts into two categories: The parts which appear once(we will call them \emph{singletons}) and the parts that appears twice(we will call them \emph{pairs}). Then, we will find the ``smallest" partition which contains a fixed number of singletons and a fixed number of pairs. (This partition is called the \emph{base partition}). Using certain moves on parts, we start from the base partition and obtain any $rrg_{3,a}$ partition as a result. This will give us a bijection. This bijection is later used to construct a new series for Rogers-Ramanujan-Gordon partitions when $k=3$.
We use square brackects, [ ], to denote pairs and paranthesis, ( ), to denote singletons.
Before discussing the details of the construction, we start with some examples.

\begin{example}
\label{ex:a=3}
Consider the partition $\lambda=14 + 14 + 11 + 10 + 7 + 7 + 5 + 5 + 2  + 1$. Note that this partition satisfies the conditions of $rrg_{3,3}$. We will use this partition to construct a partition triple $(\beta,\mu,\nu)$ where $\beta$ is the base partition, $\mu$ is the partition which contains \textit{backward moves} applied on the pairs(the parts that repeat) of $\lambda$ and $\nu$ is the partition which contains backward moves applied on the singletons(the parts that do not repeat) of~$\lambda$. 

\textbf{Step1:} Firstly, divide the parts into two: The ones which repeats and the ones which do not repeat. Thus, our partition is of the form:
$$[14,14] + (11) + (10) + [7,7] + [5,5] + (2) + (1).$$ Now, we are looking for a partition which contains $m=3$ pairs $n=4$ singletons.

\textbf{Step2:} In this step, we answer the following question: What is the smallest weight partition which contains $3$ pairs and $4$ singletons that satisfy $rrg_{3,3}$ conditions? It can be seen that this is the following partition: $$\beta = (10) + (9) + (8) + (7) + [5,5] + [3,3] + [1,1].$$ Thus, this would be our base partition.

\textbf{Step3:} We want to reach $\beta$ from $\lambda$ using backward moves. Firstly, we will obtain~$[1,1]$ using backward moves: 
\begin{align*}
    [14,14]+(11)+(10)+[7,7]+[5,5]+(2)+(1) &\xrightarrow[]{} \\ [14,14]+(11)+(10)+[7,7]+\textbf{[4,4]} + (2) + (1).
\end{align*}
Now we want to pull $[4,4]$ back and get $[3,3]$. However, the resulting partition, $[14,14]+(11)+(10)+[7,7]+\textbf{[3,3]} + (2) + (1)$, does not satisfy the $rrg_{3,3}$ conditions. Thus, we need to do adjustments on the nearby parts(the details are given in the proof of Theorem \ref{thm:main}).
\begin{align*}
    [14,14]+(11)+(10)+[7,7]+\textbf{[4,4]} + (2) + (1) &\xrightarrow[]{} \\ [14,14] + (11)+ (10)+[7,7]+(4)+(3)+\textbf{[1,1]}.
\end{align*}
We obtained $[1,1]$, as desired. Note that at each backward move on pairs, the weight of the partition decreases by $2$. We can operate inductively, i.e., $[1,1]$ will not interfere with our moves anymore. Now, we want to get $[3,3]$, the next smallest part in $\beta$:
\begin{align*}
    [14,14]+(11)+(10)+[7,7]+(4)+(3)+[1,1] &\xrightarrow[]{} \\ [14,14]+(11)+(10)+\textbf{[6,6]}+(4)+(3)+[1,1] &\xrightarrow[]{} \\ [14,14]+(11)+(10)+(6)+(5)+\textbf{[3,3]} + [1,1].
\end{align*}
Note that a similar adjustment is done in the backward move to get $[3,3]$.
Similarly, we continue with obtaining $[5,5]$:
\begin{align*}
    [14,14]+(11)+(10)+(6)+(5)+[3,3] + [1,1] &\xrightarrow[]{} \\ \textbf{[13,13]} + (11)+(10)+(6)+(5)+[3,3]+[1,1] &\xrightarrow[]{} \\ (13)+(12)+\textbf{[10,10]}+(6)+(5)+[3,3]+[1,1] &\xrightarrow[]{} \\ (13)+(12)+\textbf{[9,9]}+(6)+(5)+[3,3]+[1,1] &\xrightarrow[]{}  \\ (13)+(12)+\textbf{[8,8]}+(6)+(5)+[3,3]+[1,1] &\xrightarrow[]{} \\ (13)+(12)+(8)+(7)+\textbf{[5,5]}+[3,3]+[1,1].
\end{align*}
 We now got the pairs of $\beta$. We continue with the singletons. Note that since in the base partition the pairs comes before the singletons, we do not need to do any adjustment on backward moves on the singletons. Our first aim is to obtain $(7)$. As $(7)$ is already contained in the partition at hand, we do not perform any moves. Same arguments apply for $(8)$. Therefore, we continue with $(9)$:
\begin{align*}
    (13)+(12)+(8)+(7)+[5,5]+[3,3]+[1,1] &\xrightarrow[]{} \\ (13)+\textbf{(11)}+(8)+(7)+[5,5]+[3,3]+[1,1] &\xrightarrow[]{}  \\(13)+\textbf{(10)}+(8)+(7)+[5,5]+[3,3]+[1,1] &\xrightarrow[]{} \\ (13)+\textbf{(9)}+(8)+(7)+[5,5]+[3,3]+[1,1].
\end{align*}
Note that, each backward move on singletons decreases the weight by $1$. Lastly, we need to get $(10)$:
\begin{align*}
    (13)+(9)+(8)+(7)+[5,5]+[3,3]+[1,1] &\xrightarrow[]{} \\ \textbf{(12)}+(9)+(8)+(7)+[5,5]+[3,3]+[1,1] &\xrightarrow[]{} \\ \textbf{(11)}+(9)+(8)+(7)+[5,5]+[3,3]+[1,1] &\xrightarrow[]{} \\ \textbf{(10)}+(9)+(8)+(7)+[5,5]+[3,3]+[1,1].
\end{align*}
This yields the base partition $\beta$. We are left to construct the partitions that contains backward moves applied on the pairs, and the singletons, namely $\mu$ and $\nu$.

\textbf{Step 4: }We applied five backward moves on $[14,14]$, two backward moves on $[7,7]$, and two backward moves on $[5,5]$ where each move decreases the weight by 2. This corresponds to the partition $\mu=10 + 4 +4$. Similarly, we applied three backward moves on $(11)$, three backward moves on $(10)$, and no backward moves on $(2)$ and $(1)$. This corresponds to the partition $\nu=3+3+0+0$. As a small remark, we allow $\mu$ and $\nu$ to contain $0$ as a part. Thus, we have 
 $$14+14+11+10+7+7+5+5+2+1 \xrightarrow[]{}(10+9+8+7+5+5+3+3+1+1 \, , \, 10+4+4 \, , \, 3 + 3+0+0).$$

As a result this is the one direction of the bijection between $$rrg_{3,3} \xrightarrow[]{}(\text{Base Partition  , Backward Moves on Pairs , Backward Moves on Singletons}).$$ 
\end{example}

For the other direction of the same example, i.e., given $\beta = (10)+(9)+(8)+(7)+[5,5]+[3,3]+[1,1]$, $\mu=10+4+4$ and $\nu = 3+3+0+0$, we want to obtain $\lambda = [14,14] +(11)+(10)+[7,7]+[5,5]+2+1$. This time, we will apply \textit{forward moves} on~$\beta$, the base partition, to get $\lambda$.

\textbf{Step1:} This time, we start with the moves on the singletons unlike the other direction where we applied moves on the singletons lastly. Given $\nu = 3+3+0+0$, we need to push~$(10)$ and $(9)$ three times, and we do not push $(8)$ and $(7)$. More explicitly, 
\begin{align*}
    (10)+(9)+(8)+(7)+[5,5]+[3,3]+[1,1] &\xrightarrow[]{} \\ \textbf{(11)}+(9)+(8)+(7)+[5,5]+[3,3]+[1,1] &\xrightarrow[]{} \\ \textbf{(12)}+(9)+(8)+(7)+[5,5]+[3,3]+[1,1] &\xrightarrow[]{} \\  \textbf{(13)}+(9)+(8)+(7)+[5,5]+[3,3]+[1,1] &\xrightarrow[]{} \\(13)+\textbf{(10)}+(8)+(7)+[5,5]+[3,3]+[1,1] &\xrightarrow[]{} \\ (13)+\textbf{(11)}+(8)+(7)+[5,5]+[3,3]+[1,1] &\xrightarrow[]{} \\(13)+\textbf{(12)}+(8)+(7)+[5,5]+[3,3]+[1,1].
\end{align*}
We now got the singletons of $\lambda$.

\textbf{Step2:} We continue with the pairs. Since, $\mu = 10+4+4$, we need to apply five forward moves on $[5,5]$, two forward moves on $[3,3]$ and two forward moves on $[1,1]$. As a result,
\begin{align*}
    (13)+(12)+(8)+(7)+[5,5]+[3,3]+[1,1] &\xrightarrow[]{} \\ (13)+(12) + \textbf{[8,8]} + (6)+(5)+[3,3]+[1,1] &\xrightarrow[]{} \\ (13)+(12) + \textbf{[9,9]} + (6)+(5)+[3,3]+[1,1] &\xrightarrow[]{} \\ (13)+(12) + \textbf{[10,10]} + (6)+(5)+[3,3]+[1,1] &\xrightarrow[]{} \\ \textbf{[13,13]} + (11)+(10)+(6)+(5)+[3,3]+[1,1] &\xrightarrow[]{} \\ \textbf{[14,14]} + (11)+(10)+(6)+(5)+[3,3]+[1,1] &\xrightarrow[]{} \\ [14,14]+(11)+(10)+\textbf{[6,6]}+(4)+(3)+[1,1] &\xrightarrow[]{} \\ [14,14]+(11)+(10)+\textbf{[7,7]}+(4)+(3)+[1,1] &\xrightarrow[]{} \\ [14,14]+(11)+(10)+[7,7] +\textbf{[4,4]}+(2)+(1) &\xrightarrow[]{} \\ [14,14]+(11)+(10)+[7,7] +\textbf{[5,5]}+(2)+(1).
\end{align*}
Thus, we recover $\lambda$. Observe that the forward moves and the backward moves are exactly the opposites of each other in this example. This yields the other direction of the bijection
$${} (\text{Base Partition , Forward Moves On Pairs , Forward Moves On Singletons}) \xrightarrow[]{} rrg_{3,3}. $$ Going back to the example, we have
$$ (10 + 9 + 8 + 7 + 5 + 5 + 3 + 3 + 1 + 1 \, , \, 10 + 4 + 4 \, , \, 3 + 3+0+0) \xrightarrow[]{} 14 + 14 + 11 + 10 + 7 + 7 + 5 + 5 + 2 + 1.$$

We now turn our attention to another example where $a=2$ instead of $a=3$. In this case, the order of the moves and the definition of the moves are different. The main reason is the form of the base partition is different in this case.
\begin{example}
Let $a=2$, instead of $a=3$ as in Example \ref{ex:a=3}, and $\lambda=17 + 13 + 9 + 6 + 6 + 4 + 4 + 1$. 

Firstly, the number of pairs is $2$, i.e., $m=2$ and the number of singletons is $4$, i.e.,~$n=4$. We are looking for the smallest weight partition, the base partition, which satisfies $rrg_{3,2}$ conditions. Moreover, it should contain $2$ pairs and $4$ singletons. So, the base partition is: 
$$\beta=[8,8]+[6,6]+(4)+(3)+(2)+(1).$$ We note here that the uniqueness and the particular form of $\beta$ will become more evident by Lemma \ref{lem:shape}. Our aim is to reach from $\lambda$ to $\beta$ using backward moves.

\textbf{Step1: }Firstly, we want to obtain $(1)$. Since we already have $(1)$ in $\lambda$, there is no need to apply any backward moves. We continue with $(2)$:
    \begin{align*}
        (17)+(13)+(9)+[6,6]+[4,4]+(1) &\xrightarrow[]{} \\ (17)+(13)+\textbf{(8)}+[6,6]+[4,4]+(1) &\xrightarrow[]{} \\ (17)+(13)+[7,7]+[5,5]+\textbf{(3)}+(1) &\xrightarrow[]{} \\ (17)+(13)+[7,7]+[5,5]+\textbf{(2)}+(1).
    \end{align*}
Note that, at the backward move which turns $(8)$ into $(3)$ we need to do some adjustments on the nearby parts. Next, we get $(3)$:
\begin{align*}
    (17)+(13)+[7,7]+[5,5]+(2)+(1) &\xrightarrow[]{} \\ (17)+\textbf{(12)}+[7,7]+[5,5]+(2)+(1) &\xrightarrow[]{} \\(17)+\textbf{(11)}+[7,7]+[5,5]+(2)+(1) &\xrightarrow[]{} \\{} (17)+\textbf{(10)}+[7,7]+[5,5]+(2)+(1) &\xrightarrow[]{} \\(17)+\textbf{(9)}+[7,7]+[5,5]+(2)+(1) &\xrightarrow[]{} \\ (17)+[8,8]+[6,6]+\textbf{(4)}+(2)+(1) &\xrightarrow[]{} \\ (17)+[8,8]+[6,6]+\textbf{(3)}+(2)+(1).
\end{align*}

Lastly, we obtain $(4)$:
    \begin{align*}
        (17)+[8,8]+[6,6]+(3)+(2)+(1) &\xrightarrow[]{} \\ \textbf{(16)}+[8,8]+[6,6]+(3)+(2)+(1) &\xrightarrow[]{} \\ \textbf{(15)}+[8,8]+[6,6]+(3)+(2)+(1) &\xrightarrow[]{} \\ \textbf{(14)}+[8,8]+[6,6]+(3)+(2)+(1) &\xrightarrow[]{} \\ \textbf{(13)}+[8,8]+[6,6]+(3)+(2)+(1) &\xrightarrow[]{} \\ \textbf{(12)}+[8,8]+[6,6]+(3)+(2)+(1) &\xrightarrow[]{} \\ \textbf{(11)}+[8,8]+[6,6]+(3)+(2)+(1) &\xrightarrow[]{} \\ \textbf{(10)}+[8,8]+[6,6]+(3)+(2)+(1) &\xrightarrow[]{} \\ [9,9]+[7,7]+\textbf{(5)}+(3)+(2)+(1) &\xrightarrow[]{} \\ [9,9]+[7,7]+\textbf{(4)}+(3)+(2)+(1).
    \end{align*}
Now, all singletons in the base partition $\beta$ are obtained. Note that each backward move on the singletons decreases the weight of the partition by $1$. 

\textbf{Step2: }We continue with the pairs. First, we obtain $[6,6]$:
    \begin{align*}
        [9,9]+[7,7]+(4)+(3)+(2)+(1) \xrightarrow[]{} [9,9]+\textbf{[6,6]}+(4)+(3)+(2)+(1).
    \end{align*}
To conclude, we need to obtain $[8,8]$:
    $$[9,9]+[6,6]+(4)+(3)+(2)+(1) \xrightarrow[]{} \textbf{[8,8]}+[6,6]+(4)+(3)+(2)+(1).$$
We arrived at $\beta$. Note that each backward move on pairs decreases the weight by $2$. 

\textbf{Step3: }We are left to construct the partitions $\mu$ and $\nu$. Counting the total number of backward moves: nine backward moves on $(17)$, six backward moves on $(13)$, three backward moves on $(9)$, no backward moves on $(1)$, one backward move on $[6,6]$, and one backward move on $[4,4]$ leads to $\nu=(9,6,3,0)$ and $\mu=(2,2)$. As a result, what we have established is the correspondence:
    $$17 + 13 + 9 + 6 + 6 + 4 + 4 + 1 \xrightarrow[]{} (8+8+6+6+4+3+2+1 \, , \, 9+6+3+0 \, , \, 2+2).  $$
    For the other direction, we start from the triple partition
    $$(\beta,\mu,\nu) = ([8,8]+[6,6]+(4)+(3)+(2)+(1),\ 2+2,\ 9+6+3+0).$$
    We want to obtain $\lambda = 17 + 13 + 9 + 6 + 6 + 4 + 4 + 1$ using forward moves on $\beta$. 

\textbf{Step1: }Similar to Example \ref{ex:a=3}, we will reverse the order of the moves we applied. More precisely, we start with forward moves on pairs. Since $\mu = 2+2$, this means that we need to apply one forward move on $[8,8]$ and one forward move on $[6,6]$.
    \begin{align*}
        [8,8]+[6,6]+(4)+(3)+(2)+(1) &\xrightarrow[]{} \\ \textbf{[9,9]}+[6,6]+(4)+(3)+(2)+(1) &\xrightarrow[]{} \\ [9,9] + \textbf{[7,7]}+(4)+(3)+(2)+(1).
    \end{align*}
This concludes the part with the pairs. 

\textbf{Step2: }Next, we continue with the singletons. Firstly, we apply nine forward moves on $(4)$:
    \begin{align*}
        [9,9]+[7,7]+(4)+(3)+(2)+(1) &\xrightarrow[]{} \\ [9,9]+[7,7]+\textbf{(5)}+(3)+(2)+(1) &\xrightarrow[]{} \\ \textbf{(10)}+[8,8] + [6,6]+(3)+(2)+(1) &\xrightarrow[]{} \\ \textbf{(11)}+[8,8] + [6,6]+(3)+(2)+(1) &\xrightarrow[]{} \\ \textbf{(12)}+[8,8] + [6,6]+(3)+(2)+(1) &\xrightarrow[]{} \\ \textbf{(13)}+[8,8] + [6,6]+(3)+(2)+(1) &\xrightarrow[]{} \\ \textbf{(14)}+[8,8] + [6,6]+(3)+(2)+(1) &\xrightarrow[]{} \\ \textbf{(15)}+[8,8] + [6,6]+(3)+(2)+(1) &\xrightarrow[]{} \\ \textbf{(16)}+[8,8] + [6,6]+(3)+(2)+(1) &\xrightarrow[]{} \\ \textbf{(17)}+[8,8] + [6,6]+(3)+(2)+(1).
    \end{align*}
Similarly, we apply six moves on $(3)$:
    \begin{align*}
        (17)+[8,8]+[6,6]+(3)+(2)+(1) &\xrightarrow[]{} \\ (17)+[8,8]+[6,6]+\textbf{(4)}+(2)+(1) &\xrightarrow[]{} \\(17)+\textbf{(9)}+[7,7]+[5,5]+(2)+(1) &\xrightarrow[]{} \\(17)+\textbf{(10)}+[7,7]+[5,5]+(2)+(1) &\xrightarrow[]{} \\(17)+ \textbf{(11)}+[7,7]+[5,5]+(2)+(1) &\xrightarrow[]{} \\ (17)+\textbf{(12)}+[7,7]+[5,5]+(2)+(1) &\xrightarrow[]{} \\(17)+\textbf{(13)}+[7,7]+[5,5]+(2)+(1).
    \end{align*}
Lastly, we need to apply three forward moves on $(2)$ in order to get $\lambda$ back:
    \begin{align*}
        (17)+(13)+[7,7]+[5,5]+(2)+(1) &\xrightarrow[]{} \\ (17)+(13)+[7,7]+[5,5]+\textbf{(3)}+(1) &\xrightarrow[]{} \\ (17)+(13)+\textbf{(8)}+[6,6]+[4,4]+(1) &\xrightarrow[]{} \\ (17)+(13)+\textbf{(9)}+[6,6]+[4,4]+(1).
    \end{align*} 
Hence, we got the correspondence:
    $$ (8+8+6+6+4+3+2+1 \, , \, 9+6+3+0 \, , \, 2+2) \xrightarrow[]{} 17 + 13 + 9 + 6 + 6 + 4 + 4 + 1,$$
as desired. In the next section, we state the main theorem of the paper and give the proof that the above prodecures always work.

\end{example}

\section{Main Theorem}\label{sec:main}

Throughout this section, we take $a=3$ (recall that the number of 1's
as a part in $rrg_{k,a}$ partitions cannot exceed $a-1$), and we formally define the ``forward moves'' and ``backward moves". Then, we prove that we get the bijection introduced in Section \ref{sec:examples} using these moves. This allows us to write a new series for $rrg_{3,a}$ partitions. Moreover, at the end of the section we explain what happens if $a=2$ or $a=1$. We directly state our main theorem:
\begin{theorem}
\label{thm:main}
Let $rrg_{3,3}(m,n)$ denote the number of partitions of $n$ with exactly $m$ parts which satisfies the $rrg_{3,3}$ conditions. Then,

$$T_3(x) = \sum_{m,n \ge 0} rrg_{3,3}(m,n) x^mq^n = \sum_{m,n \ge 0} \frac{q^{4\binom{m+1}{2}+2mn+\binom{n+1}{2}-2m} x^{2m+n}} {(q^2;q^2)_m(q;q)_n} $$
\end{theorem}

To prove this theorem, we need a lemma which shows that the base partition is really what is claimed, i.e., it is of the form $(2m+n)+\ldots+(2m+3)+(2m+1)+[2m-1,2m-1]+\ldots+[3,3]+[1,1]$.

\begin{lemma}
\label{lem:shape}
    The base partition for $rrg_{3,3}$ with $m$ pairs and $n$ singletons, i.e., the partition which satisfies the $rrg_{3,3}$ conditions with $m$ pairs $n$ singletons and has the smallest weight is the following unique partition: $$(2m+n) + (2m+n - 1)+ \ldots + (2m+1) + [2m-1 , 2m-1] + [2m-3,2m-3] + \ldots +[3,3] + [1 , 1].$$
\end{lemma}

\begin{proof}    
We need to prove the followings:

\begin{enumerate}
    \item The form of the base partition is as claimed,
    \item The claimed partition is unique.
\end{enumerate}

If it is not the smallest weight partition, then there must exist a backward move on one of the parts so that the weight can be decreased. Either we can move a singleton or a pair backward. Next, we show that neither a backward singleton move nor a backward pair move is possible.

Suppose a backward move on a singleton is possible. If we want to move a singleton backward, we need to move $(2m+1)$. This is not allowed due to the presence of the pairs~$[2m-1,2m-1],\ [2m-3,2m-3],\ \ldots,\ [1,1]$ as this would lead to a partition of the form:

$$ \text{Larger Parts} +  (2m+2) + [2m,2m]+ [2m-2 , 2m-2] + \ldots +  [2,2] + (1).$$ However, $[2,2]+(1)$ violates the difference conditions of $rrg_{3,3}$. Thus, the backward move on singletons is not a possibility.

    Similarly, we cannot apply a backward move on pairs since the pairs already start from the smallest number, $[1,1]$, and the difference between consecutive pairs is the smallest as well. Thus, a backward move on pairs is not possible as well. As a result, this must be the smallest weight partition. Moreover, the form is unique, since the each consecutive difference is the smallest possible that satisfy $rrg_{3,3}$ conditions.
      
\end{proof}

\begin{proof}[Proof of the main theorem]
The proof is done in three steps. First, we will show that if forward and backward moves are defined properly, we get the claimed series. Second, the proper definition of forward moves are given. Third, we give the proper definition of backward moves.

\textbf{Step1: }We show that each partition, $\lambda$ , counted by $rrg_{3,3}(m,n)$ corresponds to a triple of partitions $(\beta , \mu , \nu) $ where $\beta$ is the corresponding base partition, $\mu$ is the partition which keeps the forward moves on pairs and $\nu$ is the partition which keeps the forward moves on singletons. Let $n_1$ be the number of pairs in $\lambda$ and $n_2$ be the number of singletons in $\lambda$. Then, by Lemma \ref{lem:shape} we have $$\beta = (2n_1+n_2-1) + \ldots + (2n_1+1) + (2n_1)  + [2n_1 -1 , 2n_1 - 1]+ \ldots + [3,3] + [1,1].$$ Note that, $\beta$ has $2n_1+n_2$ parts and $$|\beta| = 4\binom{n_1+1}{2}+2n_1n_2+\binom{n_2+1}{2} - 2n_1.$$  Since there are $n_1$ pairs and $n_2$ singletons in $\lambda$, the partition $\mu$ has $n_1$ even parts(because each forward move on pairs increase the weight of the partition by 2) and $\nu$ has $n_2$ parts. Hence, we obtain

\begin{align*}
    \sum_{\substack{n_1,n_2 \ge 0 \\ \mu , \nu \, \text{are partitions as above}}}q^{|\beta| + |\mu| + |\nu|}x^{2n_1+n_2} = \frac{q^{4\binom{n_1+1}{2}+2n_1n_2+\binom{n_2+1}{2} - 2n_1}x^{2n_1+n_2}}{(q^2;q^2)_{n_1}(q;q)_{n_2}}.
\end{align*}

Using $(\beta,\mu,\nu)$ we will get a unique partition $\lambda$ via forward moves, and given $\lambda$ we will get a unique $(\beta,\mu,\nu)$ via backward moves. As a result, we will prove that 
\begin{align*}
    \sum_{n,m \ge 0}rrg_{3,3}(m,n)x^m q^n = \frac{q^{4\binom{n_1+1}{2}+2n_1n_2+\binom{n_2+1}{2} - 2n_1}x^{2n_1+n_2}}{(q^2;q^2)_{n_1}(q;q)_{n_2}}
\end{align*}

\textbf{Step2: } We now investigate the forward moves. Given $(\beta , \mu , \nu)$, we want to push the parts of the base partition $\beta$ to get $\lambda$, a $rrg_{3,3}$ partition. We will apply the forward moves on the singletons. More precisely, we push the largest singleton in the base partition $\nu_1$ times, the second largest singleton $\nu_2$ times, $\ldots$, and the smallest singleton $\nu_{n_2}$ times. We then apply the forward moves on the pairs. Again, we push the largest pair $\mu_1$ times, the second largest pair $\mu_2$ times, $\ldots$, and the smallest pair $\mu_{n_1}$ times. The details of the moves are given as follows.
\begin{enumerate}
        \item Forward moves on singletons: Since in the base partition the pairs comes before the singletons, the forward moves on singletons are straightforward. Pushing singleton $(a)$ gives us $(a+1)$ which increases the weight of the partition by $1$.
        \item Forward moves on pairs: If there is a pair $[b,b]$ we want to push, there are two cases to consider:
        
        \begin{enumerate}
            \item The pair becomes $[b+1,b+1]$ and the resulting partition satisfies $rrg_{3,3}$ conditions. In this case, the move is defined as: $[b,b] \xrightarrow[]{} \textbf{[b+1,b+1]}$.
            \item  We cannot make it $[b+1,b+1]$ because it violates the conditions of $rrg_{3,3}$. The reason is the existence of the singleton $(b+2)$, and the possible existence of other singletons $(b+3),(b+4),\ldots,(b+s)$ for some integer $s\ge 2$. Then, the forward move on the pair is defined as
            \begin{align*}
                (b+s) + (b+s-1) + ... + (b+3)+(b+2) + [b,b] \xrightarrow[]{} \\ \textbf{[b+s,b+s]} + (b+s-2) + ... +(b+1)+(b).
            \end{align*}  
        \end{enumerate}
    \end{enumerate}
As a result, given $(\beta,\mu,\nu)$ we get a unique $\lambda$ which satisfies $rrg_{3,3}$ conditions.

\textbf{Step3: }Lastly, we need to investigate the backward moves. Given a $\lambda$ which satisfies~$rrg_{3,3}$ conditions, first we look at the number of pairs and singletons in $\lambda$. This allows us to find the corresponding base partition, $\beta$. Our aim is to pull parts of $\lambda$ and get $\beta$. We start with the smallest pair in $\lambda$ and apply backward moves on it until it becomes $[1,1]$ and continue with the next smallest pair of $\lambda$ and apply backward moves on it until it becomes $[3,3]$, and so on. Once we obtain all the pairs of the base partition $\beta$, we continue with the singletons, again from the smallest singleton to the largest singleton. In other words, we will apply backward moves on the smallest singleton in $\lambda$ until we obtain the smallest singleton in $\beta$, and so on. More precisely, the backward moves are defined as:
 \begin{enumerate}
        \item For the backward moves on pairs, there are two cases to consider: 
        \begin{enumerate}
            \item  A backward move on $[b+1,b+1]$ is defined as : $[b+1,b+1] \xrightarrow[]{} \textbf{[b,b]}$, if it does not violate the $rrg_{3,3}$ conditions. 
            \item  We cannot turn $[b+1,b+1]$ into $[b,b]$, i.e., there is a singleton $(b-1)$ and possibly $(b-2), (b-3) ,\ldots, (b-s)$, where $s$ in an integer $2 \le s < b$. In this case, we define the backward move as
            \begin{align*}
            [b+1,b+1] + (b-1)+\ldots + (b-s+1) + (b-s) \xrightarrow[]{} \\ (b+1)+(b)+\ldots+(b-s+2) + \textbf{[b-s,b-s]}.   
            \end{align*}
            Note that in either case the weight of the partition decreases by $2$.
        \end{enumerate}
        \item Backward moves on singletons: since we first pull the pairs back, the backward moves on singletons are straightforward: $(b+1)$ becomes $(b)$. Note that the weight of the partition decreases by $1$.
        \end{enumerate}
        
        Observe that, the backward moves are exactly the opposite of the forward moves.
        For each pair, we keep track of the backward moves we applied to that pair. We put this into~$\mu$. Similarly, each backward moves on singletons are stored in $\nu$. More precisely,~$\mu$ would be a vector of length $n_1$ where $\mu_i$ is the number of backward moves applied to the~$i$-th largest pair in $\lambda$. Similarly, $\nu$ would be a vector of length $n_2$ where $\nu_i$ is the number of backward moves applied to the $i$th largest singleton in $\lambda$. As a result, given $\lambda$ satisfying $rrg_{3,3}$ conditions, we get a unique $(\beta , \mu, \nu)$. This concludes the proof. \qedhere

\end{proof}

In Theorem \ref{thm:main}, we only considered the case $a=3$. It is possible to take $a=1$ or~$a=2$, as well. We discuss the details in the following remarks.

\begin{remark}
    The proof of Theorem \ref{thm:main} works for $a=1$ as well, the only difference is that the corresponding base partition is $$\beta = (2n_1+n_2) + \ldots + (2n_1+2) + (2n_1+1)  + [2n_1 , 2n_1]+\ldots + [4,4] + [2,2]$$ instead of $(2n_1+n_2-1) + \ldots + (2n_1+1) + (2n_1)  + [2n_1 -1 , 2n_1 - 1] +\ldots + [3,3] + [1,1]$. This indeed is the only difference since the form of the base partition is exactly same as in $a=3$ case. The corresponding generating function is $$T_1(x) = \sum_{n,m \ge 0} \frac{q^{4\binom{m+1}{2}+2mn+\binom{n+1}{2}+n}x^{2m+n}}{(q^2;q^2)_m(q;q)_n}$$
\end{remark}

\begin{remark}
The proof of Theorem \ref{thm:main} works for $a=2$ if the definition of the moves are changed as well as the base partition:
    \begin{enumerate}
        \item Our base partition has to change since we cannot use $[1,1]$ anymore. As a result, our base partition would be the following:
        $$ [n + 2m , n + 2m] + [n + 2m - 2 , n+2m - 2] + \ldots + [n+2 , n+2] + (n) + (n - 1)  + \ldots + (2) + (1).$$ 

        Since, now, in the base partition the singletons comes before the pairs, we need to redefine the moves as well.
        \item We look at the forward moves first: 
        \begin{enumerate}
            \item Forward moves on pairs: since pairs comes after the singletons in the base partition, the move is defined as $[b,b] \xrightarrow[]{} [b+1 , b+1]$. Note that, the weight of the partition increases by $2$.
            \item Forward moves on singletons: we have two cases to consider.
            \begin{enumerate}
                \item If one can push them, without violating the conditions, the move is defined as $b \xrightarrow[]{} b+1$. The weight of the partition increases by $1$.
                \item  Otherwise, there are some pairs which prevents our move, namely we are in the situation $ [b+2s , b+2s] + \ldots + [b+4 , b+4] + [b+2 , b+2] + (b)$. Then, the forward move on $(b)$ is
                \begin{align*}
                [b+2s , b+2s] + \ldots + [b+4 , b+4] + [b+2 , b+2] + (b) \xrightarrow[]{} \\
                \textbf{(b+2s+1)} + [b+2s-1 , b+2s-1] + \ldots +[b+3,b+3] , [b+1, b+1].
                \end{align*}
         Note that the weight of the partition increases by $1$.
            \end{enumerate}
            \item We now look at the details of backward moves.

        \begin{enumerate}
                \item Backward moves on pairs: since there cannot be any restrictions, $[b+1,b+1]$ becomes $[b,b]$. Note that the weight of the partition decreases by $2$.
                \item Backward moves on singletons: 
                \begin{enumerate}
                    \item If one can pull them without violating the conditions, then the backward move is $(b+1) \xrightarrow[]{}(b)$.
                    \item However, it can be the case that there are some restrictions due to nearby pairs: $(b+1) + [b-1,b-1] + [b-3,b-3] + [b-2s+1 , b-2s+1]$ for some integer $1 \le s < b/2$, then the backward move is defined as
                    \begin{align*}
                    (b+1) + [b-1,b-1] + [b-3,b-3]+ \ldots + [b-2s+1 , b-2s+1] \xrightarrow[]{} \\
                    [b,b]+[b-2,b-2]+\ldots+[b-2s+2,b-2s+2] + \textbf{(b - 2s)}\end{align*}
                    \end{enumerate} 
            \end{enumerate}
             
        \end{enumerate}
        \item Unlike the $a=3$ case, we first apply forward moves on pairs, and then on singletons. On the other hand, we apply backward moves on singletons first, and then on the pairs.
        \item As a result, we get the generating function $$T_2(x) = \sum_{n,m \ge 0} \frac{q^{4\binom{m+1}{2}+2mn+\binom{n+1}{2}}x^{2m+n}}{(q^2;q^2)_m(q;q)_n}.$$
    \end{enumerate}

\end{remark}

\section{Conclusions, Discussion and Future Research Directions} \label{sec:Conc}

In this paper, we consider a new series for $rrg_3$ partitions and give a combinatorial explanation using base partition and moves idea.

We conclude this paper by mentioning three research directions that naturally arise from our work:

\begin{enumerate}
    \item Generalizing our approach to cases when $k >3$. We have done preliminary work on the case $k=4$, however generalizing our approach seems to be a more challenging task since one has triples as well as singletons and pairs. As a general observation, once number of types of parts increases, defining the moves becomes more and more difficult.
   
    \item The best case scenario is to generalize our series to an arbitrary $k$ and find series of the form $$\sum_{m,n \ge 0} rrg_k(m,n)x^mq^n = \sum_{n_1,n_2,n_3,\ldots,n_{k - 1} \ge 0}\frac{q^{\text{QUADRATIC}}x^{\text{LINEAR}}}{(q;q)_{n_1}(q^2;q^2)_{n_2}\ldots(q^{k-1};q^{k-1})_{n_{k-1}}}.$$ Is it possible to combinatorially interpret this series using base partition and moves, as we discussed in this paper?
    \item The ultimate goal is: given a series of the form $$\sum_{n_1,n_2,n_3,\ldots,n_{k - 1} \ge 0}\frac{q^{\text{QUADRATIC}}x^{\text{LINEAR}}}{(q^{\alpha_1};q^{\alpha_1})_{n_1}(q^{\alpha_2};q^{\alpha_2})_{n_2}\ldots(q^{\alpha_{k-1}};q^{\alpha_{k-1}})_{n_{k-1}}}$$ is it possible to \textbf{automatically} interpret it combinatorially?  We emphasize ``automatically" since the hardest part of this approach is to define part types(singletons and pairs in our case) and the moves.  Thus, it would be very beneficial to have a combinatorial framework which works for all such series. 
\end{enumerate}
As a final remark, it should be noted that it is possible to construct our series via adding and removing staircase.
\section*{Acknowledgements}

We thank Kağan Kurşungöz for useful discussions.

\bigskip
\bibliographystyle{abbrv}
\bibliography{direkt}

\end{document}